\documentclass{amsart}
\usepackage{amsmath,amsthm,amssymb}

\newtheorem{theorem}{Theorem}[section]
\newtheorem{lemma}[theorem]{Lemma}
\newtheorem{corollary}[theorem]{Corollary}%
\newtheorem*{propositionNoNumber}{Proposition}
\theoremstyle{definition}
\newtheorem{definition}[theorem]{Definition}
\newtheorem{example}[theorem]{Example}

\newtheorem{proposition}[theorem]{Proposition}
\theoremstyle{remark}
\newtheorem{remark}[theorem]{Remark}
\numberwithin{equation}{section}
\usepackage{comment}

\usepackage[all]{xy}
\usepackage{enumerate}
\usepackage{color}
\usepackage{comment}
\usepackage{amsthm, amssymb}
\usepackage{amsmath}
\usepackage{mathtools}

\newcommand{\Rep}{\operatorname{Rep}}
\newcommand{\supp}{\operatorname{supp}}

\newcommand{\Ult}{\operatorname{Ult}}
\newcommand{\Spc}{\operatorname{Spc}}
\newcommand{\pfd}{\operatorname{pfd}}
\renewcommand{\dim}{\operatorname{dim}}
\newcommand{\Ker}{\operatorname{Ker}}
\newcommand{\Coker}{\operatorname{Coker}}
\newcommand{\otherwise}{\operatorname{otherwise}}

\newcommand{\cA}{\mathcal{A}}
\newcommand{\cF}{\mathcal{F}}
\newcommand{\cM}{\mathcal{M}}
\newcommand{\cP}{\mathcal{P}}
\newcommand{\cQ}{\mathcal{Q}}
\newcommand{\cS}{\mathcal{S}}
\newcommand{\cI}{\mathcal{I}}

\newcommand{\ZZ}{\mathbb{Z}}
\newcommand{\AZ}{A_{\ZZ}}
\newcommand{\Rfd}{\Rep^{\pfd}}
\newcommand{\RAZ}{\Rfd(\AZ)}
\newcommand{\VAZ}{\vec{A}_{\ZZ}}

\begin{document}

\title[Prime tensor ideals in abelian categories]{Prime ideals in categories of representations of quivers of type $A$}

\author{Shunsuke Tada}
\address{}
\curraddr{}
\email{kyorochan1357@gmail.com}
\thanks{}

\begin{abstract}
{We study prime tensor ideals in tensor abelian categories of quiver representations. Specifically, we classify the prime tensor ideals in the category of representations of zigzag quivers (with bounded path length) whose vertex set is the set of integers. We show that prime tensor ideals in these categories are in canonical bijection with prime ideals of a Boolean algebra, the power set of integers.}\end{abstract}

\maketitle

\section{Introduction}\label{sec1}

We will say a quiver (directed graph) is \emph{of type $A$} if its vertex set is the set of integers, and there is a unique arrow between $i$ and $i+1$ for each $i$ like below
$$\dots
    \longleftrightarrow (-3)
    \longleftrightarrow (-2)
    \longleftrightarrow (-1)
    \longleftrightarrow 0
    \longleftrightarrow 1
    \longleftrightarrow 2
    \longleftrightarrow 3 \longleftrightarrow
     \dots ,$$
where  $ \longleftrightarrow$ means that we can choose the direction of each arrow. Such quivers' representations are used to study persistent homology which has topological information about the structure of data sets, \cite{bib10}. The representations help visualize the topological information via the ``barcode", \cite{bib20}. The theory of persistent homology is used in many areas. 
For example, see \cite{bib30} and \cite{bib40} for applications to biology, \cite{bib50} for an application to chemistry, and \cite{bib60} for an application to mathematics.

The category of representations of a quiver is a tensor (monoidal) category, that is, an abelian category with a monoidal structure. In this category, we can consider prime $\otimes$-ideals (see Definition \ref{primeTensorIdeal}) 
as defined by Tobias J. Peter in \cite{bib1}. This concept is related to the tensor triangulated geometry developed by Paul Balmer \cite{bib70}. 

In this paper, we classify the prime $\otimes$-ideals of the category $\RAZ$ of point-wise finite-dimensional representations (Definition \ref{definition:pfddef}) of any quiver $\AZ$ of type $A$ which has bounded path length. That is, we have assumed that there exists $n\geq 1$ in $\mathbb{N}$ such that any path in $\AZ$ has length smaller than $n$.
These prime $\otimes$-ideals are related to the prime ideals of the ring $\prod_{i\in \ZZ} (\ZZ/{2\ZZ})$. %
We will obtain the relationship between the prime $\otimes$-ideals of $\RAZ$ and the prime ideals of $\prod_{i\in \ZZ} (\ZZ/{2\ZZ})$  using the power set $\cP(\ZZ)$ of the set of integers $\ZZ$. We will obtain following theorem.

\begin{theorem}
Let $\AZ$ be a quiver of type $A$ with bounded path length. Then, there is a one-to-one correspondence between each of the three sets
\begin{enumerate}
    \item prime tensor ideals of the category $\RAZ$, 
    \item prime ideals of the ring $\prod_{i\in \ZZ} (\ZZ/{2\ZZ})$, and 
    \item prime ideals of the boolean algebra $\cP(\ZZ)$. 
\end{enumerate}
\end{theorem}

In this theorem, we need the condition
$$
\text{bounded path length}
$$
to have above relations. Without it, we could not have the classification of the prime $\otimes$-ideals of the category $\RAZ$. For example, if the quiver has an infinite length path like
$$
 \dots\longrightarrow \diamond
    \longrightarrow \diamond
    \longrightarrow \diamond
    \longrightarrow \diamond
    \longrightarrow \diamond
    \longrightarrow \diamond
    \longrightarrow\dots ,
$$
then we will have an inclusion map from the prime ideals of $\cP(\ZZ)$ to the prime $\otimes$-ideals of $\RAZ$, but the map is not surjective. We will prove this in Corollary \ref{CounterExample}.

\textbf{Outline.} This paper is organized as follows. In Section \ref{sec2}, we will see the basics of the quiver representations and lattice theory. In Section \ref{sec3}, we see the classification of prime $\otimes$-ideals of $\RAZ$ and some topological properties of the spectrum of $\RAZ$.

\textbf{Acknowledgements.} I would like to thank Shane Kelly for suggesting me the problem, discussing it.

\section{Preliminaries} \label{sec2}

\subsection{Quiver and its representation}

Here are some basics about quivers and their representations. See \cite{bib6} for %
more %
details. Throughout, $K$ will be a field.

\begin{definition}[Quiver of type $A$] \label{definition:typeA} \ \\
A quiver $Q = (Q_0, Q_1, s, t)$ is a quadruple consisting
of two sets: $Q_0$ called the set of vertices and $Q_1$
called the set of arrows, and two maps $s$, $t$ : $Q_1 \to Q_0$. For every arrow $\alpha \in Q_1$, $s(\alpha) \in Q_0 $ and $t(\alpha) \in
Q_0$ are called $source$ and $target$ of $\alpha$, respectively. We usually denote the arrow $\alpha$ by $\alpha : s(\alpha) \to t(\alpha)$.

We will say a quiver is \emph{of type $A$} if its vertex set is the set of integers, and there is a unique arrow between $i$ and $i+1$ for each $i$.
\end{definition}

\begin{definition}[Point-wise finite-dimensional representations]\label{definition:pfddef}\ \\
Let $Q = (Q_0, Q_1, s, t)$ be a quiver. The category $\Rep(Q)$ of representations of $Q$ is defined as follows. An object is pair consisting of a family of $K$-vector spaces indexed by $Q_0$, and a family of $K$-homomorphisms indexed by $Q_1$:
$$V:=((V_a)_{a\in Q_0},\ (V_{\alpha} : V_{s(\alpha)} \to  V_{t(\alpha)} )_{\alpha \in Q_1} ).$$
    
Let $V=((V_a)_{a\in Q_0},(V_{\alpha})_{\alpha \in Q_1}) $ and $W=((W_a)_{a\in Q_0}, (W_{\alpha})_{\alpha \in Q_1}) $ be representations. A morphism in $\Rep(Q)$ from $V$ to $W$ is a family of $K$-homomorphisms
$(f_a : V_a \to  W_a)_{a\in Q_0}$ 
    satisfying  
          $$ f_{t(\alpha)} \circ V_{\alpha} = W_{\alpha} \circ f_{s(\alpha)} $$
for every arrow $\alpha \in Q_1$.

A representation $V$ of $\Rep(Q)$ is said to be \emph{point-wise finite-dimensional} if the dimension of vector space $\dim_K(V_a)$ is finite for every vertex $a$ in $Q$. We write $\Rfd(Q)$ for the full subcategory of $\Rep(Q)$ whose objects are point-wise finite-dimensional representations.
\end{definition}

\begin{definition}[Interval representation]\ \\
Let $Q = (Q_0, Q_1, s, t)$ be a quiver, and let $I=(I_0, I_1, s' ,t')$ be a subquiver of $Q$. Define $K_{I}\in \Rep(Q)$ by 
$$(K_{I})_a := \left \{ \begin{array}{cc} 
 K\ & (a\in I_0) \\ 
 0\ & (\text{otherwise}), 
\end{array} \right. 
$$  
$$(K_I)_{\alpha} := \left \{ \begin{array}{cc} 
id_K\ & (\alpha \in I_1) \\ 
0\ & (\text{otherwise}). 
\end{array} \right. 
$$
If $I$ is connected and convex (and therefore also full), then $K_I$ is called an \emph{interval representation}. 
\end{definition}

\begin{theorem}\label{proposition:allAreInterval}
Let $\AZ$ be a quiver of type $A$. Then every point-wise finite-dimensional representation of $\Rep(\AZ)$ is direct sum of interval representations.
\end{theorem} 

\begin{proof}
See  \cite{https://doi.org/10.48550/arxiv.1507.01899} for the proof.
\end{proof}

\subsection{Tensor category and its tensor ideal}

Here are some basics about tensor categories and their tensor ideals. See \cite{bib7}, \cite{bib1} for more details. 

\begin{definition}[Tensor category \cite{bib7}]\ \\
A \emph{tensor category} (or \emph{monoidal category}) is  $\nu = (\nu_0, \otimes, I, a, \ell, r)$ consisting of a category $\nu_0$, a functor $\otimes: \nu_0 \times \nu_0
\to \nu_0$ , an object $I$ of $\nu_0$ and natural isomorphisms $a_{XYZ} : (X \otimes Y )\otimes Z \to X \otimes(Y \otimes Z), \ 
\ell_X : I \otimes X \to X,\ r_X : X \otimes I
\to X,$ subject to two coherence axioms expressing the
commutativity of the following diagrams:
%
\[ \xymatrix@C=-4em@R=1em{
((W \otimes X) \otimes Y) \otimes Z \ar[rr]^a \ar[dr]^(0.7){a \otimes 1} &&
(W \otimes X) \otimes (Y \otimes Z) \ar[rr]^a &&
W \otimes (X \otimes (Y \otimes Z)) \\
& (W \otimes (X \otimes Y)) \otimes Z \ar[rr]^a &&
W \otimes ((X \otimes Y) \otimes Z) \ar[ur]_(0.7){1 \otimes a}
} \]
and
\[ \xymatrix@C=-1em@R=1em{
(X \otimes I ) \otimes Y \ar[rr]^a  \ar[dr]_(0.3){r \otimes 1}
&& X \otimes (I \otimes Y) \ar[dl]^(0.3){1 \otimes \ell}\\
& X \otimes Y 
}
\] 
\end{definition}

\begin{example} \ \\
The tuple ($\RAZ,\ \otimes,\ K_{\AZ},\ a,\ \ell,\ r \ )$ is a tensor category. Let $V,\ V',\ W$ and $W'$ be objects of $\Rep(\AZ)$. We define a functor 
$$\otimes: \RAZ \times \RAZ \to \RAZ $$
by 
$$\otimes(V, W):=((V_n \otimes_K W_n)_{n\in (\AZ)_0},\ (V_{\alpha}\otimes_K W_\alpha)_{\alpha \in (\AZ)_1 } )$$ for objects and
\begin{align*}
\otimes((f,g) &: (V, W) \to (V', W')) \\
&:= (f_n \otimes g_n)_{n \in (\AZ)_0} : ((V_n \otimes_K W_n)_{n \in (\AZ)_0} \to ((V'_n \otimes_K W'_n)_{n\in (\AZ)_0}
\end{align*}
for arrows. Then, we have natural isomorphisms $a_{XYZ} :  (X \otimes Y )\otimes Z \to X \otimes(Y \otimes Z)$, $\ell_X : K_{\AZ} \otimes X \to X $ and $ r_X : X \otimes K_{\AZ}
\to X $ for any $X,\ Y$ and $Z$ in $\RAZ$. These natural isomorphisms are subject to two coherence axioms of tensor category.
\end{example} 

\begin{definition}[Prime tensor ideal \cite{bib1}]\ \\ \label{primeTensorIdeal}
Let $\cA$ be an abelian category.  A full subcategory $\cM$ of $\cA$ is called $coherent$
if 
\begin{enumerate}
    \item it contains the zero object, 
    \item it is closed under extensions and 
    \item kernels and cokernels in $\cA$ of morphisms between objects of $\cM$ are in $\cM$.
\end{enumerate}
 Suppose now that $\cA$ is a tensor abelian category. A \emph{(coherent) tensor ideal} of $\cA$ is a coherent subcategory $\cM\subset \cA$ such that $\cM \otimes \cA  \subset \cM$. A proper  ideal $\cM \subsetneq \cA$ is called a \emph{(coherent) prime ideal} if $A \otimes B \in \cM$ implies $A\in \cM$ or $B \in \cM$. 

We write the set of all the prime tensor ideals by $\Spc(\cA)$.
\end{definition}

\begin{example}\ \\
Let $Q$ be a quiver, and consider the tensor category ($\Rep^{\pfd}(Q),\ \otimes,\ K_{Q},\ a,\ \ell,\ r \ )$ and let $\cM$ be the full subcategory whose objects are $\{V \in Q : V_n =0\}$ for some vertex $n \in Q_0$. We will show that the full subcategory $\cM$ is a prime $\otimes$-ideal.

First, we show that $\cM$ is closed under kernels and cokernels.
Let $V$ and $W$ be representations in $\cM$ and   $(f_n : V_n \to  W_n)_{n\in Q}$ be a morphism from a representation $V$ to $W$. Since $V_n$ and $W_n$ are isomorphic to zero, the kernel and the cokernel of $f_n: V_n \to W_n$ is isomorphic to zero. Thus, $\Ker((f_n)_{n\in \AZ})$ and $\Coker((f_n)_{n\in \AZ})$ is in $\cM$. So the full subcategory $\cM$ is closed under kernels and cokernels.  

Next, we show that $\cM$ is closed under extension. Take an exact sequence 
$$0 \to V' \to V \to V'' \to 0$$
where $V'$ and $V''$ are objects of $\cM$. Then, there is an exact sequence of $K$-homomorphisms
\[ 0 \to V_n' \to V_n \to V_n'' \to 0. \]
Since the representations $V'$ and $V''$ are in $\cM$, we have $ V'_n = 0 = V''_n$, so $V_n$ is isomorphic to $0$. Thus the representation $V$ is in $\cM$. So the full subcategory  $\cM$ is closed under extensions. 

Then, for any representations $V\in \Rfd(Q)$ and $ W \in \cM$, we have $(V\otimes W)_n=V_n \otimes W_n=0$. Thus $\cM$ is a $\otimes$-ideal.

Finally, we show that $\cM$ is a prime $\otimes$-ideal. Let $V$ and $W$ be the representations of $\Rep^{\pfd}(Q)$. If $V\otimes W$ is in $\cM$, then we have $0=(V\otimes W)_n=V_n\otimes W_n$. Thus $V_n$ or $W_n$ is isomorphic to $0$ and therefore $V$ or $W$ is in $\cM$. 
\end{example}

\subsection{Boolean rings and Boolean algebras}

Here are some basics about Boolean rings and Boolean algebras. See \cite{bib2} for more details.

\begin{definition}[Boolean ring]\ \\
A ring $B$ with an identity is called a $Boolean$ $ring$ if $x^2 = x$ for all ${x\in B}$.
\end{definition}


\begin{proposition}\label{remark:18} \ \\
Any Boolean ring is a commutative ring, every prime ideal of a Boolean ring is a maximal ideal, and every residue field is isomorphic to $\ZZ/2\ZZ$.
\end{proposition}

\begin{proof}
See exercise 57 on page 194 of \cite{fraleigh2020pearson} and exercise 11.(ii) on page 11 of \cite{atiyah1994introduction}. 
\end{proof}

\begin{definition}[Boolean algebra]\ \\
Let $B$ be a poset. The \emph{Boolean algebra} is defined to be a tuple $(B;\lor,\ \land,\ ',\ 0,\ 1)$ such that 
\begin{enumerate}[(i)]
    \item $x\lor y:=\text{sup}\{x,y\}$ and $x \land y:=\text{inf}\{x,y\}$ exist for all $x, y \in B,$
    \item $(B;\lor,\ \land)$ is distributive, i.e.,  
$$a\land(b\lor c)=(a\land b)\lor(a\land c),$$ 
    \item $0$ is the minimal element, and $1$ is the maximal element of $B$,
    \item every element $x$ admits a unique complement, i.e., an element satisfying $x \land x' = 0$ and $x \lor x' = 1$.
\end{enumerate}
\end{definition}

\begin{definition}[Ideal/prime ideal of Boolean algebra]\ \\
Let $B$ be a Boolean algebra. A non-empty
subset $J$ of $B$ is called an \emph{ideal} if
\begin{enumerate}[(i)]
\item $a,b \in J$ implies $a\lor b\in J$, 
\item $a\in B, b\in J$ and $a\leq b$ imply a $\in J$.
\end{enumerate}
An ideal is \emph{proper} if $J \subsetneq B$.
A proper ideal $J$ of $B$ is said to be a prime ideal if $a, b \in B$ and $a \land b \in J$
imply $a \in J$ or $b \in J$. We denote the sets of prime ideals by $\Spc(B).$
\end{definition}

\begin{example} \ \\
Let us write $\cP(\ZZ)$ for the power set of $\ZZ$. Then $ (\cP(\ZZ);\cap,\ \cup, \ ^c,\ \emptyset,\ \ZZ)$ is a Boolean algebra where $^c$  is the operation of taking the complement of an element. Then $\cP_a:=\{I\in \cP(\ZZ):a \notin I \} $ is the prime ideal for every $a \in \ZZ.$
\end{example}

\begin{proposition}\label{proposition:aaa} \ \\
\begin{enumerate}[(i)]
    \item Let $B$ be a Boolean algebra and define + and $\cdot$ on $B$ by
$$x+y :=(x \land y')\lor(x' \land y),
\qquad  
x\cdot y := x\land y.$$
Then ($B$; +, ·) is a Boolean ring.
    \item Let $B$ be a Boolean ring and define $\land$, $\lor$ and $'$  on $B$ by
$$x\lor y :=x+y+xy,
\qquad
x\land y:=xy,
\qquad
x':= 1+x.$$
Then $(B;\land,\lor,',0,1)$ is a Boolean algebra.
    \item The correspondence between Boolean algebras and Boolean rings established in (i) and (ii) is a bijective one.
    \item There is a one-to-one inclusion preserving correspondence between prime ideals of a Boolean algebra and the prime ideals of the associated Boolean ring. 
\end{enumerate}
\end{proposition}

\begin{proof}
For the proof of (i) $\sim$ (iii) see \cite[Ex.4.27]{bib2}.
As for the proof of (iv), by \cite[Ex.4.29]{bib2} and Proposition \ref{proposition:aaa} (i), (ii), and (iii), a subset  of a Boolean algebra is an ideal of the Boolean algebra if and only if the subset is an ideal of the corresponding Boolean ring. We can easily see that the correspondence of their ideals gives the correspondence of their prime ideals.

\end{proof}

By this proposition, we can see Boolean algebras and Boolean rings are equivalent.

\begin{example} \label{example:29}\ \\
Consider the Boolean algebra $(\cP(\ZZ); \cap, \cup, ^c, \emptyset, \ZZ)$ corresponding to the power set of the set of integers. The Boolean ring corresponding to $\cP(\ZZ)$ is isomorphic to $\prod_{i\in \ZZ} (\ZZ/{2\ZZ})$ via the map 
$$\phi : \cP(\ZZ) \to \prod_{i\in \ZZ} (\ZZ/{2\ZZ}), \quad  I\mapsto (a_i)_{i\in \ZZ}$$ where $a_i=1+2\ZZ$ if $i\in I$ otherwise $a_i=0+2\ZZ$. Note $\phi$ is a bijective ring homomorphism, i.e., a ring isomorphism. Here, for each $I,\ J\in \cP(\ZZ)$, the sum $I+J:=(I\cup J)\backslash (I\cap J)$ and the product $I \cdot J:=I\cap J $ on the left are those defined in Proposition \ref{proposition:aaa}(i). By Proposition \ref{proposition:aaa}(iv), The map $\phi$ gives a bijection between $\Spc(\cP(\ZZ))$ and $\Spc(\prod_{i\in \ZZ} (\ZZ/{2\ZZ}))$.
\end{example}

\section{Prime tensor ideals of $\RAZ$} \label{sec3}

\subsection{Classification of the prime tensor ideals of $\text{Rep}^{\text{pfd}}$($A_{\ZZ}$)}

Here, we will classify the prime  $\otimes$-ideals of $\Rfd(A_\ZZ)$ when the quiver $A_\ZZ$ has bounded path length.

Denote the class of prime ideals of $\Rfd(A_\ZZ)$ and $(\cP(\ZZ); \cap,\cup,  ^c, \emptyset, \ZZ)$ by $\Spc(\Rfd(A_\ZZ))$ and  $\Spc(\cP({\ZZ}))$ respectively. We will consider a subset $Q'_0$ of the vertices of a quiver $Q = (Q_0, Q_1, s, t)$ as a full subquiver of $Q$. A full subquiver is determined by the set of vertices. For example, we will be able to see $\ZZ = (\AZ)_0$ as $\AZ$. 

\begin{proposition} \label{lemma:KsuppViffV}
Assume that there is a natural number $n\geq 1$ such that the length of paths in the quiver $\AZ$ is less than $n$. Then, for any representation $V\in \Rfd(A_\ZZ)$ and a prime $\otimes$-ideal $\cM\in \Spc ( \Rfd (A_\ZZ))$, we have $K_{\supp(V)}\in \cM$ if and only if $V\in \cM$.
\end{proposition}

\begin{proof}
$(\Rightarrow)$. 
Since $\cM$ is a prime $\otimes$-ideal, we have $V=V\otimes K_{\supp(V)} \in \cM$. 
\\
$(\Leftarrow)$. Let the representation $V$ be in $\cM$. We have $0=K_{\supp(V)}\otimes K_{\supp(V)^c} \in \cM$. Since $\cM$ is a prime $\otimes$-ideal, either $K_{\supp(V)}$ or $K_{\supp(V)^c}$ are in $\cM$. We show that $K_{\supp(V)} \in \cM$ by contradiction. More precisely, we will show that $K_{\supp(V)^c} \in \cM$ implies $K_\ZZ \in \cM$, and this contradicts the assumption that $\cM$ is prime.

Since $\cM$ is a prime $\otimes$-ideal, $\cM$ is closed under extension. Thus, for an exact sequence,
$$ 0\to V \to V\oplus K_{\supp(V)^c} \to K_{\supp(V)^c} \to 0   
$$
the representation $V\oplus K_{\supp(V)^c} $ is in $\cM$. Then, we have $K'_{\ZZ} \otimes (V\oplus K_{supp(V)^c}) \in \cM$ where the representation $K'_{\ZZ}$ is defined by
$$
(K'_{\ZZ})_{a} := K  \  (\ a \ \text{is any integer})
$$
$$
(K'_{\ZZ})_{\alpha} :=0 \  (\ \alpha \ \text{is any arrow})
$$
Then, the representation $K'_{\ZZ}$ is in $\cM$ because there exists a morphism
$$
K'_{\ZZ} \otimes (V\oplus K_{\supp(V)^c}) \to K'_{\ZZ} \otimes (V\oplus K_{\supp(V)^c})
$$
such that its kernel is $K'_{\ZZ}$. Here, note that support of $K'_{\ZZ} \otimes (V\oplus K_{\supp(V)^c})$ is $\ZZ$. 

Recall that we assume that all paths of the quiver $\AZ$ have length smaller than or equal to $n$. This implies that the quiver has infinitely many sinks and sources. Take an arbitrary sink and let the sink be $t_0$ ($t$ comes from ``target"). And define a source $s_1$ as the smallest source larger than $t_0$ . Define a sink $t_2$ as the smallest sink larger than $s_1$. Inductively, we can define sinks $t_{2i}$ and sources $s_{2i+1}$ for all non negative integers $i$. Similarly, we can define sinks $t_{2i}$ and sources $s_{2i-1}$ for all non positive integers $i$. %

Let $B \subset \ZZ$ be $\bigcup_{i\in \ZZ} [s_{4i-1},s_{4i+1}]$, and $D \subset \ZZ$ be $\bigcup_{i\in \ZZ} [s_{4i+1},s_{4i+3}]$ and $C$ be $\bigcup_{i\in \ZZ} \{s_{4i-1},s_{4i+1} \}$. Then we obtain
$$K_{C}=K_B \otimes K_D
$$
If we write pictures of the sets $B$, $D$ and $C$, the sets will be the integers marked by $\underbrace{}$ where B is  like

$$
\dots \longleftarrow  \underbrace{s_{-5}
\longrightarrow t_{-4}
\longleftarrow s_{-3} }
    \longrightarrow t_{-2}
    \longleftarrow \underbrace{ s_{-1}
    \longrightarrow t_0
    \longleftarrow s_1 }
    \longrightarrow t_2
    \longleftarrow \underbrace{s_3 \longrightarrow }
     \dots
$$
and $D$ is like

$$
\dots \underbrace{
\longleftarrow  s_{-5} }
\longrightarrow t_{-4}
 \longleftarrow \underbrace{s_{-3} 
    \longrightarrow t_{-2}
    \longleftarrow  s_{-1} }
    \longrightarrow t_0
    \longleftarrow  \underbrace{s_1 
    \longrightarrow t_2
    \longleftarrow s_3 } \longrightarrow 
     \dots
$$
and $C$ is like 
$$
\dots 
 \longleftarrow  \underbrace{ s_{-5} } 
\longrightarrow t_{-4}
 \longleftarrow \underbrace{s_{-3} } 
    \longrightarrow t_{-2}
    \longleftarrow  \underbrace{s_{-1} } 
    \longrightarrow t_0
    \longleftarrow  \underbrace{s_1 } 
    \longrightarrow t_2
    \longleftarrow \underbrace{s_3}  \longrightarrow \dots
$$

Since $K_B \otimes K_D=K_{C}=K_{C}\otimes K'_{\ZZ} \in \cM$, we have $K_B$ or $K_D$ are in $\cM$. We will assume that $K_B$ is in $\cM$. (we can assume $K_D \in \cM$ instead of $K_B \in \cM$. The proof is similar.  ) 
\\

For this $B$, we will show the subsets $B'$, $B''$ of $\ZZ$ where $B':=B \cup \bigcup_{i\in \ZZ} [s_{4i+1}, t_{4i+2}-1]$ and $B'':=B' \cup \bigcup_{i\in \ZZ} [t_{4i-2}+1, s_{4i+2}-1]$ satisfy $K_{B'}\in \cM$ and $K_{B''}\in \cM$. We will show that $K_{\ZZ} \in \cM$ using $K_{B''}\in \cM$.\\

(Proof of  $K_{B'}\in \cM$) \\
For the $B$, we show $K_{B'} \in \cM$ where $B':=B \cup \bigcup_{i\in \ZZ} [s_{4i+1}, t_{4i+2}-1]$. 
Define a subset $B_1:=B\cup \{s_{4i+1}+1 : i\in \ZZ, \ s_{4i+1}+1 \ \text{is not a sink} \}$. Then, we have an exact sequence 
$$ 0 \to K_{B_1 \setminus B} \to K_{B_1} \to K_B \to 0  
$$
where $K_{B_1 \setminus B}= K_{B_1 \setminus B} \otimes K'_{\ZZ} \in \cM$ and since $\cM$ is closed under extensions, the representation $K_{B_1}$ is in $\cM$. Inductively, we can define subsets of $\ZZ$ by
$$B_{\ell+1}:=B_{\ell}\cup \{s_{4i+1}+(\ell+1) \  : \  i\in \ZZ, \ s_{4i+1}+1, \dots, s_{4i+1}+\ell  \ \text{are not sinks} \}$$
for non negative integer $\ell$. Then we can show that $K_{B_{\ell}} \in \cM$ by mathematical induction.
For $\ell=1$, we have already shown that $K_{B_1}\in \cM$.
If we have $K_{B_{\ell}}\in \cM$, then we can show $K_{B_{\ell+1}}\in \cM$ because for $K_{B_{\ell}} \in \cM$, there is an exact sequence
$$ 0 \to K_{B_{\ell+1} \setminus B_{\ell}} \to K_{B_{\ell+1}} \to K_{B_{\ell}} \to 0
$$
We obtain $K_{B_{\ell+1}}$ in $\cM$ because $\cM$ is closed under extensions with $K_{B_{\ell}}$ and $K_{B_{\ell+1}\setminus B_{\ell}} \in \cM$. Thus, by mathematical induction, we have $K_{B_{\ell}} \in \cM$ for any $\ell \geq 1$. Note that there exists $m \leq n$ such that $B_{m} = B_{m+1}=\dots$ because the length of paths are smaller than or equal to $n$. Then, the set $B_{m}$ contains $s_{4i+2}-1$ for any $i\in \ZZ$ and we have $B'=B_m$ and $K_{B'} \in \cM$.
\\

(Proof of  $K_{B''}\in \cM$) \\
 For the $B'$ defined above, we show $K_{B''} \in \cM$ where $B'':=B' \cup \bigcup_{i\in \ZZ} [s_{4i-2}+1, s_{4i+2}-1]$. 
 
Define a subset $B'_1:=B'\cup \{s_{4i-1}-1 : i\in \ZZ, \ s_{4i-1}-1 \ \text{is not a sink} \}$. Then, we have an exact sequence 
$$ 0 \to K_{B'_1 \setminus B'}  \to K_{B'_1} \to K_{B'} \to 0  
$$
where $K_{B'^1 \setminus B'}=K_{B'^1 \setminus B'} \otimes K'_{\ZZ} \in \cM$ and since $\cM$ is closed under extensions, $K_{B'_1}$ is in $\cM$. Inductively, we can define subsets of $\ZZ$ as
$$B'_{(\ell+1)}:=B'_{\ell}\cup \{s_{4i-1}-(\ell+1) : i\in \ZZ, \ s_{4i-1}-1, \dots, s_{4i-1}-(\ell+1)  \ \text{are not sinks} \}$$
for non negative integer $\ell$. Then we can show that $K_{B_{\ell}} \in \cM$ for any non-negative integers by mathematical induction.
For $\ell=1$, we have already shown that $K_{B'_1}\in \cM$.
If we have $K_{B'_{\ell}}\in \cM$, then we can show $K_{B'_{(\ell+1)}}\in \cM$ because for $K_{B'_{\ell}} \in \cM$, there is an exact sequence
$$ 
0\to K_{B'_{(\ell+1)} \setminus B'_{\ell} }  \to K_{B'_{(\ell+1)}} \to K_{B'_{\ell}} \to 0
$$

We obtain $K_{B'_{(\ell+1)}}$ in $\cM$ because $\cM$ is closed under extensions. Thus, by mathematical induction, we have $K_{B'_{\ell}} \in \cM$ for any $\ell \geq 1$. Note that there exists $m \leq n$ such that $B'_{m} = B'_{(m+1)}=\dots$ because the length of paths are smaller than or equal to $n$. Then, the set $B'_{m}$ contains $t_{4i-2}+1$ for any $i\in \ZZ$ and we have $B''=B'_m$ and $K_{B''} \in \cM$.
\\

(Proof of $K_{\ZZ} \in \cM$)\\
Define $E = \{t_{4i+2} : i \in \ZZ \}$ and note that $B''=\ZZ \setminus \{t_{4i+2} : i \in \ZZ \} = \ZZ \setminus E$ and $K_E = K_E \otimes K'_\ZZ \in \cM$. We have the exact sequence
$$
0\to K_E \to K_{\ZZ} \to  K_{B''} \to 0.
$$

Since $\cM$ is closed under extensions and $K_{B''}, K_E \in \cM$, the representation $K_{\ZZ} \in \cM$. Thus, $K_{\ZZ} \in \cM$ and it leads $\cM=\RAZ$. This is a contradiction. Hence, we have $K_{\supp(V)}\in \cM$.
\end{proof}

\begin{proposition} \label{lemma:31} \ \\
There is a bijection between $\Spc(\Rfd(A_\ZZ))$ and   $\Spc(\cP(\ZZ))$.
\end{proposition}
\begin{proof} Define two maps 
\begin{align*}
\phi: \Spc(\Rfd(A_\ZZ)) 
&\to  \Spc(\cP({\ZZ})),  \\
\cM 
&\mapsto \{\supp(V) \in \cP(\ZZ): V\in \cM \}, \\ 
\psi:  \Spc(\cP({\ZZ})) 
&\to \Spc(\Rfd(A_\ZZ)), \\
\cQ 
&\mapsto \{V\in \Rfd(A_\ZZ) : \supp(V)\in \cQ \}.
\end{align*}

First, we check  $\phi(\cM) \in \Spc(\cP({\ZZ}))$ for $\cM \in \Spc(\Rfd(A_\ZZ))$ and $ \psi(\cQ)\in \ \Spc(\Rfd(A_\ZZ))$ for $\cQ\in \Spc(\cP({\ZZ}))$. After that, we will show $\phi$ and $\psi$ are inverse maps of each other.

The following (i), (ii) and (iii) show that $\phi(\cM)$ is an ideal of $\cP({\ZZ})$ and (iv) shows that $\phi(\cM)$ is a prime ideal of $\cP({\ZZ})$.

(i) $\phi(\cM)$ is not empty because $\emptyset = \text{supp}(0) \in \phi(\cM)$ where $0$ is the zero representation.

(ii) For any $\supp(V)$ and $\supp(V') \in \phi(\cM)$ (where $V$ and $V'$ $\in \cM$), we will show $\supp(V) \cup \supp(V') \in \phi(\cM)$. Since $\cM$ is a prime $\otimes$-ideal of 
$\Spc(\cP({\ZZ}))$, for the exact sequence
$$ 0\to V \to V\oplus V'\to V'\to 0,
$$ the direct sum $V\oplus V'$ of $V$ and $V'$ is in $\cM$. Thus, we have  $\supp(V) \cup \supp(V')= \supp(V\oplus V')\in \phi(\cM).$

(iii) For any $I\in \cP(\ZZ)$ and for any $\supp(V) \in \phi(\cM)$, we will show that if $I \subset \supp(V)$, then $I \in \phi(\cM)$. Assume that $I \subset \supp(V)$. Then we have $\supp(K_I)=\supp(K_I \otimes V)$. Since $K_I \otimes V \in \cM$, we obtain  $K_I \in \cM $ and $I\in \phi(\cM)$. 
 
By (i), (ii) and (iii), $\phi(\cM)$ is an ideal of $\cP(\ZZ)$.

(iv) If $I\cap J \in \phi(\cM)$ for $I$, $J$ $\in \cP(\ZZ)$, we will show that $I$ or $J$ are in $\phi(\cM)$. Since $I\cap J \in \phi(\cM)$, there exists $W$  $\in \cM$ such that $I\cap J=\supp(W)$. Define $V\in \RAZ$ by 
$$
V_{a} := 
\left \{ 
\begin{array}{ll} 
W_a \  & a\in \supp(W) \\ 
K & a\in I \setminus \supp(W) \\
0 & \otherwise
\end{array} 
\right . 
$$
$$
V_{\alpha} := 
\left \{ 
\begin{array}{ll} 
W_{\alpha} \  & s(\alpha), t(\alpha)\in \supp(W) \\ 
0 & \otherwise
\end{array} 
\right . 
$$
where $a$ is a vertex and $\alpha$ is an arrow of $\AZ$.

Then we have $V\otimes K_{\supp(W)}=W\in \cM$. Since $\cM$ is a prime $\otimes$-ideal, the representations $V$ or $K_{\supp(W)}$ are in $\cM$. If $V\in \cM$, then we have $I=\supp(V) \in \phi(\cM)$. On the other hand, if $K_{\supp(W)}\in \cQ$, then $K_{I}\otimes K_{J}=  K_{\supp(W)} \in \cM$. Therefore $K_{I}$ or $K_{J}$ are in $\cM$ because $\cM$ is a prime $\otimes$-ideal. This implies $I$ or $J$ are in $\phi(\cM).$   

Next, the following (i), (ii) and (iii) show that $\psi(\cQ)$ is a $\otimes$-ideal of $\RAZ$ and (iv) shows that $\psi(\cQ)$ is a prime $\otimes$-ideal of $\RAZ$.

(i) We will show that $\psi(\cQ)$ is closed under kernels and cokernels. For any $f: V\to W$ in $\psi(\cQ)$, the kernel and the cokernel of $f$ satisfy
$\supp(\Ker(f))\subseteq \supp(V)$ and $\supp(\Coker(f))\subseteq \supp(W)$ where $\supp(V)$ and $\supp(W)$ are in the prime ideal $\cQ$ of $\cP(\ZZ)$. Thus, we have $\supp(\Ker(f)) \in \cQ$ and $\supp(\Coker(f)) \in \cQ$. It means that $\Ker(f)$ and $\Coker(f)$ are in $\psi(\cQ)$. Hence, $\psi(\cQ)$ is closed under kernels and cokernels.

(ii) We will show that $\psi(\cQ)$ is closed under extensions. For $V',V'' \in \psi(\cQ)$, take an exact sequence 
$$ 0 \to  V' \to  V \to  V'' \to 0 $$ where $V\in \RAZ$. Then $\supp(V)=\supp(V')\cup \supp(V'')\in \cQ$ holds because $\supp(V')$ and $\supp(V'')$ are in the prime ideal $\cQ$. we have $V\in \psi(\cQ)$ and therefore, $\psi(\cQ)$ is closed under extensions.

(iii) We will show that for any $M\in \RAZ$ and $V\in \psi(\cQ)$, the representation $M\otimes V$ is in $\psi(\cQ)$. Since $\supp(M\otimes V)=\text{supp}(M) \cap \text{supp}(V) \in \cQ$ holds, we have $M\otimes V \in \psi(\cQ)$. 

By (i), (ii), and (iii), we have shown that $\psi(\cQ)$ is a $\otimes$-ideal of $\RAZ$. Now we show that it is prime.

(iv) If $V\otimes W \in \psi(\cQ)$ for $V$ and $W$ $\in \RAZ$, then supp$(V)\cap \supp(W)$=supp$(V\otimes W)\in \cQ$. Since $\cQ$ is a prime ideal of $\cP(\ZZ)$, we have $\supp(V)$ or $\supp(W)\in \cQ $. Hence, $V$ or $W \in \psi(\cQ)$ and therefore $\psi(\cQ)$ is a prime $\otimes$-ideal of $\cP(\ZZ)$.  

Finally, we show $\phi$ and $\psi$ are inverse maps of each other.
For any $\cM\in \Spc(\Rfd (A_\ZZ)) $ and $\cQ\in \Spc(\cP(\ZZ))$, we have


\begin{align*}
\phi \circ \psi(\cQ) 
&=\phi(\{V\in \Rfd(\ZZ):\supp(V)\in \cQ \}) \\
&=\{\supp(V) \in \cP(\ZZ):V\in \{V'\in \Rfd(A_\ZZ):\supp(V')\in \cQ \} \} \\
&=\{\supp(V) \in \cP(\ZZ): V \in \Rfd(A_\ZZ),\  \supp(V)\in \cQ \} \\
&=\cQ.
\end{align*}
The last equality holds because every $I \subseteq \ZZ$ is the support of some representation (e.g., $K_I$).

\begin{align*}
\psi \circ \phi(\cM) 
&=\psi(\{\supp(V)\in \cP(\ZZ):V\in \cM \}) \\
&=\{V\in \Rfd(A_\ZZ) : \supp(V)\in \{\supp(V')\in \cP(\ZZ):V'\in \cM \} \} \\
&=\{V\in \Rfd(A_\ZZ) : \supp(V) = \supp(V')\ \exists\ V'\in \cM \} \\
&=\cM.  \qquad (\because \text{Proposition \ref{lemma:KsuppViffV}} \ )
\end{align*}
Here, we use Proposition \ref{lemma:KsuppViffV} to show the relation
$$\cM \supseteq \{V\in \Rfd(A_\ZZ) : \supp(V) = \supp(V')\ \exists\ V'\in \cM \}.
$$
We will prove it. For any $V\in\{V\in \Rfd(A_\ZZ) : \supp(V) = \supp(V')\ \exists\ V'\in \cM \}$, there exists $V'\in \cM$ such that $\supp(V)=\supp(V')$. By Proposition \ref{lemma:KsuppViffV}, the interval representation $K_{\supp(V')} \in \cM$. Thus, we obtain $K_{\supp(V)}=\supp(V)\otimes K_{\supp(V')} \in \cM $. Therefore, the representation $V=V \otimes K_{\supp(V)}$ is in $\cM$. We have shown 
$$\cM = \{V\in \Rfd(A_\ZZ) : \supp(V) = \supp(V')\ \exists\ V'\in \cM \}.
$$
The two maps $\phi$ and $\psi$ are inverse maps of each other. Therefore there is a bijection between $\Spc(\RAZ)$ and $\Spc(\cP(\ZZ)) $.
\end{proof}

The next corollary is our main theorem.

\begin{corollary}

There is a one-to-one correspondence between the following three sets, $ \Spc (\cP(\ZZ))$, $\Spc (\prod_{i\in \ZZ} (\ZZ/{2\ZZ}))$ and $\Spc (\RAZ)$.
\end{corollary}
\begin{proof}
There is a bijection between $\Spc(\cP(\ZZ))$ and $\Spc(\prod_{i\in \ZZ} (\ZZ/{2\ZZ}))$ by Example \ref{example:29}. There is a bijection between $\Spc(\Rfd(A_\ZZ))$ and   $\Spc(\cP(\ZZ))$ by Proposition \ref{lemma:31}.
\end{proof}


\subsection{Topological properties of Spc($\text{Rep}^{\text{pfd}}$($A_{\ZZ}$)}

In this subsection, we will see that $\Spc(\Rfd(A_\ZZ))$ is a compact Hausdorff space homeomorphic to the topological space $\Spc(\cP(\ZZ))\cong \Spc{(\prod_{i\in \ZZ} (\ZZ/{2\ZZ}))}$
when the quiver $A_\ZZ$ has bounded length of paths. That is, we show the following proposition.

\begin{proposition} \label{homeo}\ \\
Let $\cP(\ZZ)$ be the power set of $\ZZ$ and consider it as a Boolean ring (isomorphic to $(\prod_{i\in \ZZ} (\ZZ/{2\ZZ}))$ by Example \ref{example:29}). Then we have a homeomorphism between $\Spc (\RAZ)$ and $\Spc(\cP(\ZZ))$, where both sets are equipped with their respective Zariski topologies (defined below).
\end{proposition}

We begin with the definition of the Zariski topology on $\Spc(\cA)$ for a general tensor abelian category $\cA$.

\begin{definition}[Zariski topology on tensor categories \cite{bib1}, \cite{bib70}] \label{ZarTopTen} \ \\
Let $\cA$ be a tensor abelian category. For every family of objects $ \cS \subset \cA $ we denote by $Z( \cS)$  the following subset of  $\Spc(\cA)$:
$$
Z(\cS):=\{\cM \in \Spc{\cA}: \cS\cap \cM = \emptyset \}.
$$
We will see the collection $\{Z( \cS):  \cS \subset \cA \}$ defines closed subsets of $\Spc(\cA)$ at Proposition \ref{ZAISKI}. We call the topological space \emph{Zariski topology}. The subsets $Z(\{ V\})=\{\cM\in \Spc(\cA): V  \notin \cM\}$ for $V \in \cA$ form a basis of closed subsets.
\end{definition} 

\begin{proposition}\label{ZAISKI}  \ \\
The collection of sets $Z(\cS)$ from Definition~\ref{ZarTopTen} are the closed subsets of a topology.
\end{proposition}

\begin{proof}
Let $( \cS_j)_{j\in J}$ be a collection of families of objects of $\cA$ and let $ \cS_1$ and $ \cS_2$ be the families of objects of $\cA$. Then we have $\bigcap_{j\in J}Z( \cS_j)=Z( \bigcup_{j\in J} \cS_j )$ and $Z( \cS_1)\cup Z( \cS_2)=Z(\{ V_1\oplus V_2 : V_1 \in  \cS_1, \ V_2 \in  \cS_2  \})$. Moreover, we have $Z(\cA)=\emptyset$ and $Z(\emptyset)=\cA$. Thus the collection  $\{Z( \cS):  \cS \subset \cA \}$ are the closed subsets of a topology on $\Spc(\cA)$.
\end{proof}

\begin{lemma}\label{Clopen} \ \\
Consider the tensor category $\Rfd(A_\ZZ)$, where $A_\ZZ$ is a quiver of type $A$ with bounded path length. %
Then, for any representation $V\in \Rfd(A_\ZZ)$ the closed set $Z(\{V\})=\{\cM\in \Spc(\Rfd(A_\ZZ)): V  \notin \cM\}$ is open. 
\end{lemma}

\begin{proof}
For any prime $\otimes$-ideal $\cM$ in $\Spc(\Rfd(A_\ZZ))$, the representation $K_{\supp{V}}$ is in $\cM$ if and only if $K_{\ZZ \setminus \supp{V}}$ is not in $\cM$. This is because if $K_{\supp{V}} \in \cM$ and $K_{\ZZ \setminus \supp{V}} \in \cM$, then we have $K_{ \supp{V}}\oplus K_{\ZZ \setminus \supp{V}} \in \cM $. It means $K_{\ZZ}= K_{\supp(K_{ \supp{V}}\oplus K_{\ZZ \setminus \supp{V}})} \in \cM$ by Lemma \ref{lemma:KsuppViffV}. This contradicts that $\cM$ is a prime $\otimes$-ideal. Thus $K_{\supp{V}}$ or $K_{\ZZ \setminus \supp{V}}$ is not in $\cM$. Also, one of the representations $K_{\ZZ \setminus \supp{V}}$ and $K_{\supp{V}}$ is included. This is because $K_{\supp{V}} \otimes K_{\ZZ \setminus \supp{V}} =0$ in $\cM$ and since $\cM$ is the prime $\otimes$-ideal, either $K_{\supp{V}}$ or $K_{\ZZ \setminus \supp{V}}$ is in $\cM$. Thus we have shown that the representation $K_{\supp{V}}$ is in $\cM$ if and only if $K_{\ZZ \setminus \supp{V}}$ is not in $\cM$.

Now by Lemma \ref{lemma:KsuppViffV}, we have the equation $Z(\{V\}):=\{\cM: V\notin \cM \} =\{\cM: K_{\supp{V}}\notin \cM \}=:Z (\{K_{\supp{V}}\})$. Thus we have 
\begin{align*}
Z(\{V\})^c  
&=Z(\{K_{\supp{V}}\})^c \\
&:= \{\cM\in \Spc(\Rfd(A_\ZZ)): K_{\supp{V}}  \notin \cM\}^c\\
&=\{\cM\in \Spc(\Rfd(A_\ZZ)): K_{\supp{V}}  \in \cM\} \\ 
&=\{\cM\in \Spc(\Rfd(A_\ZZ)): K_{\ZZ \setminus \supp{V}}  \notin \cM\} \\
&\qquad \qquad(\because K_{\supp(V)}\in \cM \ \text{iff} \  K_{\ZZ\setminus \supp(V)}\notin \cM) \\
&=: Z(\{K_{\ZZ \setminus \supp{V}}\}). 
\end{align*}
Therefore $Z(\{V\})$ is open.
\end{proof}

\begin{corollary}\label{Hdff}
Consider the tensor category $\Rfd(A_\ZZ)$, where $A_\ZZ$ is a quiver of type $A$ with bounded path length. %
Then, the topological space $\Spc(\Rfd(A_\ZZ))$ is a Hausdorff space.
\end{corollary}

\begin{proof}
Let $\cM_1$ and $\cM_2$ be different prime $\otimes$-ideals. Then there exists a representation $V$ such that $V$ is in $\cM_1 \setminus \cM_2$. Then  $\cM_2 \in Z(\{V\})$ and  $\cM_1 \in Z(\{V\})^c$. By Lemma \ref{Clopen}, the closed set $Z(\{V\})$ is open. Thus $\Spc(\Rfd(A_\ZZ))$ is a Hausdorff space.
\end{proof}


\begin{lemma}\label{CPT}
Let $R$ be a commutative ring. Then $\Spc(R)$ is compact as the Zariski topology.
\end{lemma}

\begin{proof}
We can refer to \cite{atiyah1994introduction} Chapter 2, Exercise 17 (v).
\end{proof}

We rewrite the statement of Proposition \ref{homeo}  again. 

\begin{propositionNoNumber}[\ref{homeo}] \ \\
Let $\cP(\ZZ)$ be the power set of $\ZZ$ and consider it as a Boolean ring (isomorphic to $(\prod_{i\in \ZZ} (\ZZ/{2\ZZ}))$ by Example \ref{example:29}). Then we have a homeomorphism between $\Spc (\RAZ)$ and $\Spc(\cP(\ZZ))$, where both sets are equipped with their respective Zariski topologies.
\end{propositionNoNumber}

\begin{proof}
We will show that the bijective map 
\begin{align*}
    \psi:  \Spc(\cP({\ZZ})) 
&\to \Spc(\Rfd(A_\ZZ)), \\
\cQ 
&\mapsto \{V\in \Rfd(A_\ZZ) : \supp(V)\in \cQ \}
\end{align*}
defined in Proposition \ref{lemma:31} is a homeomorphism. Since we know that $\Spc(\cP(\ZZ))=\Spc(\prod_{i\in \ZZ} (\ZZ/{2\ZZ}))$ is compact (see Lemma \ref{CPT}) and $\Spc(\Rfd(A_\ZZ))$ is a Hausdorff space (see Lemma \ref{Hdff}), it is enough to show that the bijective map $\psi:  \Spc(\cP({\ZZ})) \to \Spc(\Rfd(A_\ZZ))$ is continuous (a continuous bijective map from compact space to a Hausdorff space is homeomorphism, for example, \cite[Proposition 13.26]{Topsp} for the proof).

Let $Z( \cS)=\bigcap_{W\in  \cS}Z(\{W\})$ be a closed set of $\Spc(\Rfd(A_\ZZ))$. Then we obtain $\psi^{-1}(Z( \cS))=\bigcap_{W\in  \cS}{\psi^{-1}(Z(\{W\}))}$. For each $\psi^{-1}(Z(\{W\}))$, we have the equation 

\begin{align*}
    \psi^{-1}(Z(\{W\}))
&=\psi^{-1}(\{\cM\in \Spc(\Rfd(A_\ZZ)): W  \notin \cM\})\\
&=\psi^{-1}(\{\cM\in \Spc(\Rfd(A_\ZZ)): K_{\supp(W)}  \notin \cM\})  \qquad ( \because \text{by Proposition \ref{lemma:KsuppViffV})} \\
&=\{\psi^{-1}(\cM) \in \Spc(\mathcal{P(\ZZ)}):   K_{\supp(W)}  \notin \cM\}\\
&=\{\psi^{-1}(\cM) \in \Spc(\mathcal{P(\ZZ)}): \supp(W) \notin \psi^{-1}(\cM) \}\\
&=\{\cQ \in \Spc(\mathcal{P(\ZZ)}): \supp(W) \notin \cQ \} \qquad \qquad   \quad( \text{Replace $\psi^{-1}(\cM)$ with $\cQ$}.)\\
&=\{\cQ \in \Spc(\mathcal{P(\ZZ)}): \ZZ \setminus \supp(W) \in \cQ \} \qquad \qquad (\because \cQ \  \text{is a prime ideal})\\ 
&=\{\cQ \in \Spc(\mathcal{P(\ZZ)}): \ZZ \setminus \supp(W) \notin \cQ \}^c\\
&={D_{\ZZ \setminus \supp(W)}}^c.\\
\end{align*}
Since $D_{\ZZ \setminus \supp(W)}$ is open, the complement ${D_{\ZZ \setminus \supp(W)}}^c=\psi^{-1}(Z(\{W\}))$ is a closed set. Thus we can conclude $ \psi^{-1}(Z( \cS))=\bigcap_{W\in  \cS}{\psi^{-1}(Z(\{W\}))}$ is a closed set and therefore the map $\psi$ is a homeomorphism.

\end{proof}

\subsection{About the representations of the quiver $\vec{\AZ}$}

In the proof of Proposition \ref{lemma:KsuppViffV}, we need the condition 
$$ \text{the length of paths are bounded.} $$
Without it, we can not prove $K_{\supp(V)}\in \cM$ if and only if $V\in \cM$ for any $V\in \Rfd(A_\ZZ)$ and $\cM\in \Spc ( \Rfd (A_\ZZ))$. For example, there may be a representation such that its support is $\ZZ$, but the representation is contained in a prime $\otimes$-ideal. We will provide such a representation when $\AZ$ is the quiver  
$$
 \dots\longrightarrow \diamond
    \longrightarrow \diamond
    \longrightarrow \diamond
    \longrightarrow \diamond
    \longrightarrow \diamond
    \longrightarrow \diamond
    \longrightarrow\dots.
$$
which will be denote by $\VAZ$. The proof seems to generalise, to any quiver of type $A$ containing an unbounded path, but we restrict our attention to $\VAZ$.

Using this, we show that the map
\begin{align*}
    \psi:  \Spc(\cP({\ZZ})) 
&\to \Spc(\Rfd(A_\ZZ)), \\
\cQ 
&\mapsto \{V\in \Rfd(A_\ZZ) : \supp(V)\in \cQ \}
\end{align*}
which appeared in the proof of Proposition \ref{lemma:31} is injective but not surjective.

We will prepare some lemmas and definitions to show the existence of the representation whose support is $\ZZ$.

\begin{definition} \label{generatedideal}
(An ideal generated by a set) \ \\
Let $Q$ be a quiver and $S$ a subset of objects of $\Rfd(Q)$ which contains the representation $0$. %
Set $S_{-1} := S$, and for $n \geq 0$ define full subcategories $S_n, S_n', S_n''$ of $\Rfd(Q)$ inductively as follows.
$$
S_n:= \left \{V\in \Rfd(Q) : 
\begin{array}{c}
\text{ there exists an exact sequence } \\
0 \  \to V' \to V \to V'' \to 0 \\
\text{ with }\ V',\ V'' \in S_{n-1} 
\end{array} \right \},
$$
$$
S'_n:=\{W\otimes V : W \in  \Rfd(Q)  , \ V \in S_n \},
$$
$$
S''_n:=\bigcup_{V,W \in S'_n}  \{ \Ker(f),\  \Coker (f) :  V \xrightarrow{f} W \}.
$$
%
%
We then define an ideal $\langle S \rangle$  of $\Rfd(Q)$ by
$$
\langle S \rangle:=\bigcup_{n\in \ZZ_{\geq 0}} S_n
$$
and call it the $\otimes$-$ideal$ $generated$ $by$ $S$. Here, we have 

$$
S \subseteq S_0 \subseteq S'_0 \subseteq S''_0 \subseteq S_1 \subseteq \dots  \subseteq S_n \subseteq S'_n \subseteq S''_n \subseteq S_{n+1} \subseteq \dots \subseteq \langle S \rangle.
$$
\end{definition} 

\begin{proposition}
The $\langle S \rangle:=\bigcup_{n\in \ZZ_{\geq 0}} S_n$ defined above
is a $\otimes$-ideal of $\Rfd(Q)$. In fact, it is the smallest $\otimes$-ideal containing the set $S$.
\end{proposition}

\begin{proof}
First we show that $\langle S \rangle$ is closed under extensions. Let $V'$ and $V''$ be representations in $\langle S \rangle$. Then, there exists $n\in \ZZ_{\geq 0}$ such that $V'$ and $V''$ are in $S_n$. Take an arbitrary exact sequence
$$
0 \to V' \to V \to V'' \to 0 ,
$$
and we have $V\in S_{n+1} \subseteq \langle S \rangle$. Thus $\langle S \rangle$ is closed under extensions. 

Next, we show that $\langle S \rangle$ is closed under kernels and cokernels. Let  $f: V \to W$ be a morphism in $\langle S \rangle$. Then, there exists $n\in \ZZ_{\geq 0}$ such that $V$ and $W$ are in $S_n$ and its kernels and cokernels are in $S_{n+1} \subseteq \langle S \rangle $.

Finally, we show that $\langle S \rangle$ is closed under tensor products. Let $V$ and $W$ be in $\Rfd(Q)$ and $\langle S \rangle$ respectively. There exists $n\in \ZZ_{\geq 0}$ such that $W \in S_n$.  We have $V\otimes W \in S_{n+1} \subseteq \langle S \rangle $. Therefore $\langle S \rangle$ is a $\otimes$-ideal of $\Rfd(Q)$.
 
Next, we will show that the $\langle S \rangle$ is the smallest $\otimes$-ideal containing $S$. Let $\cI$ be a $\otimes$-ideal containing $S$. By the construction of $S_n$ and the definition of $\otimes$-tensor ideal,  $S_n$ is included by $\cI$ for every $n\in \ZZ_{\geq0}$. Thus we have $\langle S \rangle=\bigcup_{n\in \ZZ_{\geq 0}} S_n \subseteq \cI$ and therefore $\langle S \rangle$ is the smallest $\otimes$-tensor ideal containing $S$.

\end{proof}

\begin{lemma} \label{KUKANNOYATU}
Let $[a,b]$ be an interval in $\ZZ$. If there exists a monomorphism 
$$
\phi:K_{[a,b]}\to \bigoplus_{\lambda \in \Lambda} K_{I_{\lambda}},
$$
then there exists $\lambda \in \Lambda$ such that $[a,b] \subseteq I_{\lambda}$ where $K_{[a,b]}$ and $\bigoplus_{\lambda \in \Lambda} K_{I_{\lambda}}$ are representations of $\Rfd(\Vec{A_{\ZZ}})$.

Similarly, if there exists an epimorphism
$$
\psi:\bigoplus_{\lambda \in \Lambda} K_{I_{\lambda}} \to K_{[a,b]},
$$
then there exists $\lambda \in \Lambda$ such that $[a,b] \subseteq I_{\lambda}$ where $K_{[a,b]}$ and $\bigoplus_{\lambda \in \Lambda} K_{I_{\lambda}}$ are representations of $\Rfd(\Vec{A_{\ZZ}})$.

\end{lemma}

\begin{proof}
We will show the monomorphism case by contradiction. The other case can be shown in the same way.    

Assume that for any $\lambda \in \Lambda$, we have $[a,b] \nsubseteq I_{\lambda}$. Then we obtain non-commutative diagram 
\[
  \xymatrix{
    (K_{[a,b]})_a \ar[r]_-{\phi_a} \ar[d]_{id_{K}} & (\bigoplus_{\lambda \in \Lambda} K_{I_{\lambda}})_a \ar[d]_{0} \\
    (K_{[a,b]})_b \ar[r]_{\phi_b} & ( \bigoplus_{\lambda \in \Lambda} K_{I_{\lambda}})_b \ ,
  }
\]
and it is the contradiction. Therefore there must be $\lambda \in \Lambda$ such that $[a,b] \subseteq I_{\lambda}$. 
\end{proof}

\begin{lemma}\label{KUKANNOYATU2}
Let $I$ be an interval in $\ZZ$. If there exists a monomorphism 
$$
\phi:K_{I}\to \bigoplus_{\lambda \in \Lambda} K_{J_{\lambda}},
$$
then there exists $\lambda \in \Lambda$ such that $I \subseteq J_{\lambda}$ where $K_{I}$ and $\bigoplus_{\lambda \in \Lambda} K_{I_{\lambda}}$ are  representations of $\Rfd(\Vec{A_{\ZZ}})$.

Similarly, if there exists an epimorphism
$$
\psi:\bigoplus_{\lambda \in \Lambda} K_{J_{\lambda}} \to K_I,
$$
then there exists $\lambda \in \Lambda$ such that $I \subseteq J_{\lambda}$ where $K_{I}$ and $\bigoplus_{\lambda \in \Lambda} K_{I_{\lambda}}$ are the representations of $\Rfd(\Vec{A_{\ZZ}})$.
\end{lemma}

\begin{proof}
We will show the monomorphism case. The other case can be shown in the same way.

If $I$ is bounded, then we have already shown in Lemma \ref{KUKANNOYATU}. So we will assume that $I$ is not bounded.

In case of $I=\ZZ$, let $I_n$ be the interval $[-n,n]$ for $n\in \ZZ_{\geq0}$. By restricting the interval $I=\ZZ$ to $I_n$, we obtain the monomorphism
$$
\phi:K_{I_n}\to \bigoplus_{\lambda \in \Lambda} K_{J_{\lambda}\cap I_n}.
$$
By Lemma \ref{KUKANNOYATU}, there exists $\lambda \in \Lambda$ such that $I_n \subseteq J_{\lambda}\cap I_n$. Thus for any $n\in \ZZ_{\geq0}$ we obtain $I_n \subseteq J_{{\lambda_n}}$ for some $\lambda_n$. Thus the set $C_n:=\{J_{\lambda}: I_n \subseteq J_{\lambda},\  \lambda \in \Lambda  \} \subseteq \{J_{\lambda}: 0 \in J_{\lambda},\ \lambda \in \Lambda \}$ is not empty and finite because the cardinarity of the set $\{J_{\lambda}: 0 \in J_{\lambda},\  \lambda \in \Lambda \}$ is equal to or smaller than the finite dimension  dim$_K(\bigoplus_{\lambda \in \Lambda} K_{J_{\lambda}})_0$. Thus, for the descending sequence
$$  C_0 \supseteq C_1 \supseteq C_2 \supseteq C_3 \supseteq \cdots , 
$$
there exists $m \in \ZZ_{\geq 0}$ such that:
$$
C_m=C_{m+1}=\cdots.$$
We can conclude that there exists $\lambda \in \Lambda$ such that $I=\ZZ=J_{\lambda} \in \bigcap_{n\in \ZZ_{\geq0}}C_n$.   

If $I$ is bounded above and unbounded below, then there exists $a\in \ZZ$ such that $I=(-\infty, a]$. Let $I'_n$ be the interval $[a-n,a]$ for $n \in \ZZ_{\geq0}$. We can show that $C'_n:=\{J_{\lambda}: I'_n \subseteq J_{\lambda},\ \lambda \in \Lambda \}$ is not empty because we have a monomorphism
$$
\phi:K_{I'_n}\to \bigoplus_{\lambda \in \Lambda} K_{J_{\lambda}\cap I'_n},
$$
such that there exists $\lambda \in \Lambda$ which satisfies $I'_n \subset J_{\lambda}\cap I'_n \subseteq J_{\lambda}$
by Lemma \ref{KUKANNOYATU}. Also, we can show $C'_n\subseteq \{J_{\lambda}: a \in J_{\lambda},\  \lambda \in \Lambda \} $ is finite because $\{J_{\lambda}: a \in J_{\lambda},\  \lambda \in \Lambda \}$ is equal to or smaller that the dim$_K(\bigoplus_{\lambda \in \Lambda} K_{J_{\lambda}})_a$.  We will obtain the descending sequence that stops for a $m\in \ZZ_{\geq0}$
$$ C'_0 \supseteq C'_1 \supseteq C'_2 \supseteq C'_3 \supseteq \cdots \cdots \supseteq C'_m =C'_{m+1}= \cdots.
$$
We can conclude that there exists $\lambda \in \Lambda$ such that $I'=(-\infty,a]=J_{\lambda} \in \bigcap_{n\in \ZZ_{\geq0}}C'_n$.

In case of $I$ is bounded below and unbounded above, the proof is similar to the case that $I$ is bounded above and unbounded below.
\end{proof}

\begin{lemma}\label{KUKANNOYATU3}
Let $V$ and $W$ be representations of $\Rfd(\Vec{A_{\ZZ}})$ whose interval decomposition are $\bigoplus_{\lambda \in \Lambda } K_{I_{\lambda}}$ and $\bigoplus_{\lambda' \in \Lambda' } K_{J_{\lambda'}}$. If there exists a monomorphism
$$\phi: V \to W,
$$
then for any interval $I_{\lambda}$, there exists $\lambda' \in \Lambda'$ such that $I_{\lambda} \subseteq J_{\lambda'}$.

Similarly if there exists an epimorphism
$$\psi: V \to W,
$$
then for any interval $I_{\lambda}$, there exists $\lambda' \in \Lambda'$ such that $I_{\lambda} \subseteq J_{\lambda'}$.
\end{lemma}

\begin{proof}
We will show the monomorphism case. The other case can be shown in the same way.

Since we have the natural monomorphism
$$
K_{I_{\lambda}}  \to \bigoplus_{\lambda \in \Lambda } K_{I_{\lambda}}\cong  V \xrightarrow{\phi} W \cong \bigoplus_{\lambda' \in \Lambda' } K_{J_{\lambda'}},   
$$
by Lemma \ref{KUKANNOYATU2}, there exists $\lambda' \in \Lambda'$ such that $I_{\lambda} \subseteq J_{\lambda'}$.
\end{proof}

\begin{corollary}\label{DirectsumTensor}
Let $V$ and $W$ be representations of $\Rfd(\Vec{A_{\ZZ}})$ whose interval decomposition are $\bigoplus_{\lambda \in \Lambda } K_{I_{\lambda}}$ and $\bigoplus_{\lambda' \in \Lambda' } K_{J_{\lambda'}}$. If there exists a monomorphism (resp. epimorphism)
$$\phi: V \to W,
$$
then $V$ (resp. $W$) is a direct summund of $V\otimes W$.
\end{corollary}

\begin{proof}
Now we have $V\otimes W \cong (\bigoplus_{\lambda \in \Lambda } K_{I_{\lambda}}) \otimes (\bigoplus_{\lambda' \in \Lambda' } K_{J_{\lambda'}})
\cong \bigoplus_{\lambda \in \Lambda, \lambda' \in \lambda'} K_{I_{\lambda}} \otimes K_{J_{\lambda'}}
=\bigoplus_{\lambda \in \Lambda, \lambda' \in \lambda'} K_{I_{\lambda} \cap J_{\lambda'}}$.

If the map $\phi$ is a monomorphism, for any $\lambda \in \Lambda$ there exists $\lambda' \in \Lambda'$ such that $I_{\lambda} \subseteq J_{\lambda'}$ by Lemma \ref{KUKANNOYATU3}. For these $\lambda$ and $\lambda'$, we have $K_{I_{\lambda}}\otimes K_{J_{\lambda'}}=K_{I_{\lambda} \cap J_{\lambda'}}=K_{I_{\lambda}}$. Therefore $V\otimes W \cong \bigoplus_{\lambda \in \Lambda, \lambda' \in \Lambda'} K_{I_{\lambda} \cap J_{\lambda'}}$ has the direct summand $\bigoplus_{\lambda \in \Lambda } K_{I_{\lambda}} \cong V$. 

In the similar way, we can show the case of $\phi$ is epi.
\end{proof}

\begin{lemma} \label{boundedinterval}
Let $V'$ and $V''$ be representations in $\Rfd(\VAZ)$. If these are decomposed into interval representations whose intervals are bounded, then for any exact sequence 
$$
0 \to V' \to V \to V'' \to 0,
$$
the representation $V$ is decomposed into interval representations whose intervals are bounded. That is, if $V'$ and $V''$ are of the form $\oplus_{j \in J} K_{I_j}$ for some collection $\{I_j\}_{j \in J}$ of bounded $I_i \subseteq \ZZ$, then $V$ is also of this form.   
\end{lemma}

\begin{proof}
We prove this lemma by a contradiction. Note that $V$ is of the form $\oplus_{j \in J} K_{I_j}$ for some collection $\{I_j\}_{j \in J}$ of $I_i \subseteq \ZZ$ by Proposition~\ref{proposition:allAreInterval}. We will show that if some $I_i$ is unbounded, then there is a contradiction. Assume that there is an unbounded interval $I$ in $\VAZ$  such that the interval representation $K_I$ is a direct summand of $V$ and choose an isomorphism $V\cong W\oplus K_I$. Then we can rewrite the exact sequence by
$$
0 \to V' \xrightarrow{(\phi,\psi )} W\oplus K_I \xrightarrow{(f,g)} V'' \to 0
$$
Now, assume that $I$ is unbounded above. (We can also assume $I$ is unbounded below, but the proof is similar.)
Then, $\psi_a =0$ for every $a \in \ZZ$. This is because if there exists $a \in \ZZ$ such that  $\psi_a \neq 0$, we obtain non-commutative square
\[
  \xymatrix{
    V'_a \ar[r]_-{(\phi,\psi )_a} \ar[d]_{} & (W\oplus K_I)_a \ar[d] \\
    V'_{a{+}m} \ar[r]_{(\phi,\psi )_{a{+}m}} & ( W\oplus K_I)_{a+m}
  }
\]
where $m$ is large enough for $V'_a \to V'_{a+m}$ to be zero. Thus $\psi_a =0$ for all $a \in \ZZ$.
Then, since $0=$Im$(\psi)= \Ker(g) $, we have that $g: K_I \to V'' $ is injective. This implies $V''$ has a unbounded interval representation as a direct summand. However, the representation $V''$ is decomposed into interval representations whose intervals are bounded. It is a contradiction. Hence, we can conclude that $V$ is decomposed into interval representations whose intervals are bounded.
\end{proof}

\begin{proposition}
Let $\cI$ be the full subcategory of $\Rfd(\vec{A_{\ZZ}})$ whose representations are decomposed into bounded interval representations $\oplus_{j \in J} K_{I_j}$ for some collection $\{I_j\}_{j \in J}$ of bounded $I_i \subseteq \ZZ$. Then $\cI$ is a tensor ideal of $\Rfd(\vec{A_{\ZZ}})$.
\end{proposition}

\begin{proof}
We have already shown that $\cI$ is closed under extensions in Lemma \ref{boundedinterval}. Moreover $\cI$ is closed under kernels and cokernels by Lemma \ref{KUKANNOYATU3}. 

Next, we will show that the representation $V\otimes W$ is in $\cI$ for any $V \in \Rfd(\vec{A_{\ZZ}})$ and $W\in \cI$. Let $V\in \RAZ$ and $W\in \cI$ be isomorphic to the representations $\bigoplus_{\lambda \in \Lambda } K_{I_{\lambda}}$ and $\bigoplus_{\lambda' \in \Lambda' } K_{J_{\lambda'}}$ where $J_{\lambda'} \subseteq \ZZ$ is bounded for every $\lambda'\in \Lambda'.$  Then we obtain $(\bigoplus_{\lambda \in \Lambda } K_{I_{\lambda}}) \otimes (\bigoplus_{\lambda' \in \Lambda' } K_{J_{\lambda'}}) 
\cong \bigoplus_{\lambda \in \Lambda, \lambda' \in \lambda'} K_{I_{\lambda} \cap J_{\lambda'}}\in \cI$. 

We have shown that $\cI$ is a tensor ideal.
\end{proof}

\begin{definition}
Define the representation $K'_{\ZZ} \in \RAZ$  by
$$
(K'_{\ZZ})_{a} := K, \qquad
(K'_{\ZZ})_{\alpha} :=0, \  
$$
for every $a\in \ZZ$ and every arrow $\alpha$.
\end{definition}

\begin{lemma}
Let $S$ be the set $\{0, K'_{\ZZ} \}$ of objects of $\Rfd(\VAZ)$ and let $\langle S \rangle$ be the $\otimes$-ideal generated by $S$.
Then we have $K_{\ZZ} \notin\langle S \rangle$.
\end{lemma}

\begin{proof}
We show $K_{\ZZ} \notin \langle S \rangle$ by a contradiction. Assume that $K_{\ZZ} \in \langle S \rangle$. Then there exists $n \in \ZZ_{\geq 0}$ such that $K_{\ZZ} \in S_n$ where $S_n$ is defined in Definition \ref{generatedideal}. By Lemma \ref{boundedinterval}, we can see that any representations of $S_0$ are decomposed into bounded interval representations. Since any representations of $S'_0$ and $S''_0$ are obtained by tensoring and taking kernels and cokernels, their representations are decomposed into bounded interval representations. Therefore any representation in $S_1$ are also decomposed into bounded interval representations. Inductively, every representation of $S_n$ are decomposed into bounded interval representations. Since $K_{\ZZ} \in S_n$, the interval representation  $K_{\ZZ}$ are decomposed into bounded interval representations. However, the representation $K_{\ZZ}$ is the indecomposable interval representation whose interval is infinite length. This is a contradiction. Hence we have  $K_{\ZZ} \notin \langle S \rangle$.
\end{proof}

\begin{proposition}\label{ExistMaxideal}
There exists a maximal $\otimes$-ideal in $\Rfd(\VAZ)$ that contains a representation whose support is $\ZZ$.
\end{proposition}

\begin{proof}
Let $S$ be the set $\{0, K'_{\ZZ} \}$ and $\langle S \rangle$ be the ideal generated by $S$. Define a partially ordered set $(\Sigma,\subseteq):=\{ \text{ proper ideals containing $\langle S \rangle$ } \}$. Then $\Sigma$ is not an empty set because $\langle S \rangle \in \Sigma$. Let $\sigma \subseteq \Sigma$ be a totally ordered subset. Then $ \bigcup_{\mathcal{I}\in \sigma}\mathcal{I} $ is a proper ideal (it does not contain $K_\ZZ$) and upper bound of $\sigma$. Thus by Zorn's lemma, there exists a maximal ideal containing $K'_{\ZZ}\in \langle S \rangle$ whose support are $\ZZ$. \end{proof}

\begin{lemma}\label{MAXisPRIME}
Let $\cM$ be a maximal $\otimes$-ideal of $\cA = \Rfd(\vec{A_{\ZZ}})$. Then, the maximal $\otimes$-ideal is a prime $\otimes$-ideal. 
\end{lemma}

\begin{proof}
Let $\cM$ be a maximal $\otimes$-ideal and suppose that $W\otimes W' \in \cM$. We will show that if $W\notin  \cM$ then $W' \in  \cM$. Consider the set 
$$\mathcal{I}:=\{V \in \cA : V\otimes W' \in \cM \}.$$
Then $W\in \mathcal{I}$ and $\cM \subseteq \mathcal{I}$ hold. Since we have assumed $W\notin  \cM$, the set $\mathcal{I}$ is strictly larger than $\cM$. If we show $\mathcal{I}$ is a $\otimes$-ideal, then the tensor unit $I$ is in $\cA = \mathcal{I}$ by the maximality of $\cM$, and therefore, we can show $W'= I \otimes W' \in \cM.$ So it is enough to show that $\mathcal{I}$ is a $\otimes$-ideal to prove $\cM$ is a prime $\otimes$-ideal.

First, we will show that $\mathcal{I}$ is closed under extensions. Let $V'$ and $V''$ be in $\mathcal{I}$. For any short exact sequence
$$0 \to V' \to V \to V'' \to 0,
$$
there exists an exact sequence
$$0 \to V'\otimes W' \to V\otimes W' \to V'' \otimes W' \to 0
$$
such that $V'\otimes W'$ and $V''\otimes W'$ are in the $\otimes$-ideal $\cM.$ Thus, the representation $V\otimes W'$ is in $\cM$ and therefore $V$ is in $\mathcal{I}$.

Next, we will show that $\mathcal{I}$ is closed under multiplication. Let $V$ be in $\cA$ and $V'$ be in $\mathcal{I}$. We have $(V\otimes V')\otimes W' = V\otimes (V'\otimes W') \in \cM$ and therefore $(V\otimes V') \in \mathcal{I}.$

Then, we will show that $\mathcal{I}$ is closed under kernels and cokernels. Let $f:V\to V'$ be in $\mathcal{I}$. Then, we have the monomorphism $ ker(f)\otimes id_{W'} :\Ker(f)\otimes W' \hookrightarrow   V\otimes W' $ such that $V\otimes W'$ is in $\cM$. Thus  $(\Ker(f)\otimes W')\otimes(V\otimes W')$ is in $\cM$. By Lemma \ref{DirectsumTensor}, we have  $\Ker(f)\otimes W'$ as a direct summand of $(\Ker(f)\otimes W')\otimes(V\otimes W')$ and since $\cM$ is a $\otimes$-ideal, the representation $\Ker(f)\otimes W'$ is in $\cM$. Hence $\Ker(f)$ is in $\mathcal{I}$ and therefore $\mathcal{I}$ is closed under kernels. Similarly we can show that $\mathcal{I}$ is closed under cokernels. 

We have showed $\mathcal{I}$ is a $\otimes$-ideal. 
\end{proof}

Next corollary is what we wanted.

\begin{corollary}\label{CounterExample}
There exists a prime $\otimes$-ideal of $\Rfd(\vec{A_{\ZZ}})$ which contains a representation whose support is $\ZZ$.

Thus the map 
\begin{align*}
    \psi:  \Spc(\cP({\ZZ})) 
&\to \Spc(\Rfd(A_\ZZ)), \\
\cQ 
&\mapsto \{V\in \Rfd(A_\ZZ) : \supp(V)\in \cQ \}
\end{align*}
constructed in Proposition \ref{lemma:31} is not surjective.
\end{corollary}

\begin{proof}
By Proposition \ref{ExistMaxideal}, there exists a maximal $\otimes$-ideal of $\Rfd(\vec{A_{\ZZ}})$ that contains a representation whose support is $\ZZ$ and the maximal $\otimes$-ideal is prime $\otimes$-ideal by Corollary \ref{MAXisPRIME}.  
\end{proof}


\bibliographystyle{unsrt} 
\bibliography{main.bib}
\end{document}